\documentclass[12pt, reqno]{amsart}
\usepackage[utf8]{inputenc}
\usepackage{amsmath}
\usepackage{amsthm}
\usepackage{amssymb}
\usepackage{cite}
\usepackage{tikz}
\usepackage{xcolor}
\usepackage{mathtools}
\usepackage[margin=1.4in]{geometry}

\newcommand{\upth}{\textnormal{th}}
\newcommand{\upst}{\textnormal{st}}
\newcommand\omitstuff[1]{}
\newcommand{\deriv}{\frac{d}{dq}\Bigr|_{q=1}}

\theoremstyle{plain}
\newtheorem{theorem}{Theorem}
\newtheorem{lemma}[theorem]{Lemma}

\theoremstyle{definition}

\newtheorem{observation}[theorem]{Observation}

\newtheorem{exampleblock}[theorem]{Example}

\def\E{\mathbb{E}}

\newcommand{\Sym}{\mathfrak{S}}
\title{On the average number of cycles in conjugacy class products}
\author{Jesse Campion Loth}
\address{Simon Fraser University, Burnaby, Canada}
\email{jesse\textunderscore campion\textunderscore loth@sfu.ca}
\author{Amarpreet Rattan}
\address{Simon Fraser University, Burnaby, Canada}
\email{rattan@sfu.ca}

\date{\today}

\begin{document}


\begin{abstract}
We show that for the product of two fixed point free conjugacy classes, the
average number of cycles is always very similar.  Specifically, our main result is that for a
randomly chosen pair of fixed point free permutations of cycle types $\alpha$ and $\beta$, the average
number of cycles in their product is between $H_n-3$ and $H_n+1$, where $H_n$ is
the harmonic number.
\end{abstract}

\maketitle

\section{Introduction}
Let $n$ be a positive integer.  For any finite set $A$, let $\Sym_A$ be the
symmetric group on $A$. We also use the notation $[n] := \{1, \ldots, n\}$, and $\Sym_{n} : =
\Sym_{[n]}$.  A \emph{partition of $n$} is a list $\alpha = (\alpha_1, \ldots,
\alpha_k)$ of positive integers such that $\alpha_i \geq
\alpha_{i+1}$ and $\sum_i \alpha_i = n$.  Each $\alpha_i$ is called a
\emph{part}, and the \emph{length} of $\alpha$, denoted $\ell(\alpha)$, is the
number of parts (thus $\ell(\alpha) = k$ above).  For each $0 \leq j \leq \ell(\alpha)$, we define $\alpha'_j
= \sum_{i=1}^j \alpha_i$, with $\alpha'_0 = 0$.   We let $C_\alpha$ be the conjugacy of $\Sym_n$ indexed
by $\alpha$.  We call the permutation $(1\; \cdots\; \alpha_1') \cdots
(\alpha_{\ell(\alpha) - 1}' + 1\; \cdots\; n)$ the \emph{canonical permutation 
of type} $\alpha$.   We further let $H_n : = \sum_{i=1} \tfrac{1}{i}$
be the \emph{$n^\upth$ harmonic number}.  We use the convention that
permutations are multiplied from left to right.

We start by defining the main statistic of study.  For two partitions $\alpha,
\beta \vdash n$, let $C_{\alpha, \beta}$ be the random variable for the number
of cycles in $\sigma \cdot \omega$, where $\sigma$ and $\omega$ are chosen
uniformly at random from the conjugacy classes $C_\alpha$ and $C_\beta$,
respectively. The goal of this manuscript is to estimate the expected value of
$C_{\alpha,\beta}$ when $\alpha$ and $\beta$ are partitions
without parts equal to 1.  We start by defining a class of maps $m_\pi$ for $\pi
\in \Sym_n$, and then we show how the distribution of faces of $m_\pi$ over $\pi
\in \Sym_n$ allows us to estimate $C_{\alpha, \beta}$.

    Let $\alpha$ and $\beta$ be partitions without parts of size 1 of the same
	positive integer $n$,
	and
	let $\sigma_0$ and $\omega_0$ be the canonical permutations of type $\alpha$
	and $\beta$, respectively.
    For a permutation $\pi \in \Sym_n$, we
	define a \emph{map}\footnote{Usually maps are defined more generally than we have
		presented here.  See \cite{cori1992maps}.  However, we are only interested in
	this restricted class of maps.} $m_\pi = (D, R, E(\pi))$, where $D$ is a set and $R$
	and $E(\pi)$ are permutations on $D$, which are given as follows. The set
	$D$, whose members are known as \emph{darts}, is defined by
	\begin{equation}\label{eq:defst}
		D := S \cup T, \textnormal{ where } S: = \{ s_1, \dots, s_n\} \textnormal{ and } T: = \{t_1, \dots, t_n \}. 
	\end{equation}
The set of darts can be thought of as a set of half-edges in a graph whose
left side (right side) has $\ell(\alpha)$ ($\ell(\beta)$) vertices, and the
$i^\upth$ vertex on the left side (right side) has the darts $s_{\alpha'_i + 1},
\ldots, s_{\alpha'_{i+1}}$ ($t_{\beta'_i +1}, \ldots, t_{\beta'_{i+1}})$ at it.  The permutation
    $E(\pi)$ is a fixed point free involution on the set $D$ given by
\begin{equation}\label{eq:defepi} 
        E(\pi) := (s_1 \, t_{\pi(1)}) (s_2 \, t_{\pi(2)}) \dots (s_n \,
	t_{\pi(n)}).
\end{equation} 
    Thus $E(\pi)$, also known as the \emph{edge scheme},  is in $\Sym_{D}$ and gives a perfect matching between the sets $S$
	and $T$;  the set $E(\pi)$ can be thought of as edges, where, for each $i$, we join the
	darts $s_i$ and $t_{\pi(i)}$ (half-edges) to form a complete edge.  The edge scheme along with the vertices thus
	give a bipartite graph $G$.   The permutation
	$R \in \Sym_{D}$ sends 
\begin{equation}\label{eq:compatible}
	s_{i} \rightarrow s_{\sigma_0(i)} \textnormal{ and }
	t_{i} \rightarrow  t_{\omega_0(i)};  
\end{equation}
that is,

\begin{equation}\label{eq:defrot}
    \begin{multlined}
	R := (s_1 \, s_2 \dots s_{\alpha_1'}) (s_{\alpha_1' +1}, \dots
		s_{\alpha_2'}) \dots (s_{\alpha_{\ell(\alpha)-1}' + 1}, \dots, s_n)
		\cdot\\
		(t_1 \, t_2 \dots t_{\beta_1'}) (t_{\beta_1' +1}, \dots
		 t_{\beta_2'}) \dots (t_{\beta_{\ell(\beta)-1}' + 1}, \dots, t_n),
	\end{multlined}
\end{equation}
    and its cycle type is $\alpha \cup \beta$, where $\alpha
    \cup \beta$ is the partition of $2n$ with parts in $\alpha$ or $\beta$.
    The permutation $R$, also called the \emph{rotation scheme}, encodes an
    embedding of the bipartite graph $G$; when drawn on an orientable surface,
	the cycles of $R$ give the local
    clockwise order in which the darts appear at each vertex.
	Thus the map $m_\pi = (D, R, E(\pi))$ is abstractly defined via
\eqref{eq:defst}, \eqref{eq:defepi} and \eqref{eq:defrot}, but it can be
visualized as a graph embedded on a surface.  In fact, the rotation and edge
schemes uniquely determine a \emph{2-cell embedding} of the graph
	on some orientable surface\footnote{A \emph{2-cell embedding} of a
		connected graph on
		an orientable surface is an embedding of a graph on a surface
		such that edges do not cross, and the removal of the graph from
		the surface leaves connected components each homeomorphic to a
	disc.  When the permutations $R$ and $E(\pi)$ generate an intransitive
	subgroup of $\Sym_D$, the associated graph is disconnected.   In that
	case, we take each connected component, each of which corresponds to an
	orbit of the action of the group generated by $R$ and $E(\pi)$ on $D$, and we embed
	each component as a 2-cell embedding on a suitable orientable surface.
}.
	Thus the faces (and genus 
	via Euler's characteristic formula) of the embedding are completely determined by the rotation and edge
	schemes.  The correspondence between maps and graphs embedded on
	surfaces is expounded in Cori and Machi \cite[Section
		1]{cori1992maps} and Tutte \cite{mapstutte}\footnote{Both \cite{mapstutte} and
			\cite{cori1992maps} consider the action on darts in the
			order $E(\pi)$ and then $R$, opposite to our convention.
			This produces the obvious modification to how faces are
			traversed by $R$ and $E(\pi)$.  In \cite{mapstutte}, the
			author calls maps whose underlying graphs are not
		connected \emph{premaps}.}.   In light of this correspondence, we
	define the \emph{vertices, edges} and \emph{faces} of the map $m_\pi$  to be the
	vertices, edges and faces, respectively, of the associated embedded graph.

	The faces of $m_\pi$ can easily be obtained from $R$ and $E(\pi)$;  in fact,
	they correspond to the cycles of $R \cdot E(\pi)$, and this can be seen as
	follows.  Analyzing the
	effect of the product $R \cdot E(\pi)$ on $s_i$, the permutation $R$ rotates $s_i$ to $s_{\sigma_0(i)}$; then $E(\pi)$ sends $s_{\sigma_0(i)}$ along the
	edge $\{s_{\sigma_0(i)}, t_{\pi(\sigma_0(i))}\}$ to $t_{\pi(\sigma_0(i))}$.
	It follows that the repeated effect of $R$ and $E(\pi)$ on $s_i$ is 
	\begin{equation}\label{eq:connection}
		s_i \stackrel{R}{\rightarrow} s_{\sigma_0(i)}
		\stackrel{E(\pi)}{\rightarrow} t_{\pi( \sigma_0(i))}
		\stackrel{R}{\rightarrow} t_{\omega_0(\pi( \sigma_0(i)))}
		\stackrel{E(\pi)}{\rightarrow} s_{\pi^{-1}(\omega_0(\pi(
		\sigma_0(i))))}.
	\end{equation}
	Thus we see that $R \cdot E(\pi)$ traces out the walk $s_i \rightarrow t_{\pi(
	\sigma_0(i))}  \rightarrow s_{\pi^{-1}(\omega_0(\pi( \sigma_0(i))))}
	\rightarrow \cdots$ along a face of the embedding of $m_\pi$, and that
	$(s_i\, t_{\pi( \sigma_0(i))}\,  s_{\pi^{-1}(\omega_0(\pi( \sigma_0(i))))}
	\cdots)$ is a cycle of $R \cdot E(\pi)$.  This
	correspondence is also explained in depth in \cite{cori1992maps, mapstutte}, while Example \ref{ex:mapproduct} contains a full illustration.     We note that the graph associated
	to the map $m_\pi$ may be disconnected, in
which case we treat each connected component (separately) as a 2-cell
embedding.  In that case, we follow the convention that the faces in
$m_\pi$ is the union of the faces of the connected components. This
convention preserves the correspondence between cycles of $R \cdot E(\pi)$ and
faces of $m_\pi$.  It follows that the vertices, edges and faces of $m_\pi$
correspond to the cycles of $R$, the cycles $E(\pi)$, and the cycles of $R
\cdot E(\pi)$, respectively.  In particular, we have
\begin{equation}\label{eq:facescycles}
	\# \textnormal{ of faces of } m_\pi = c(R \cdot E(\pi)).
\end{equation}
\begin{figure} 
	\centering
	\label{fig:mapproduct}
	\includegraphics[scale=1]{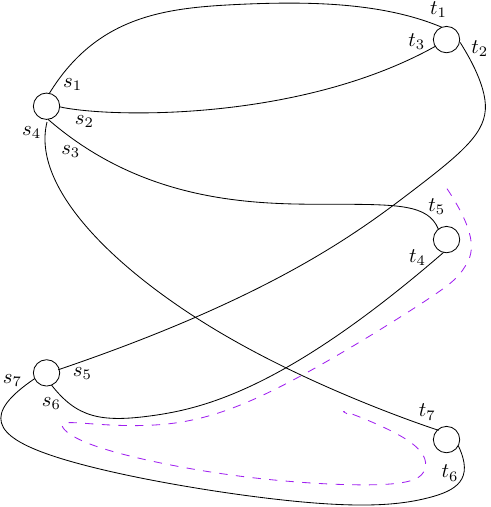}
	\caption{Pictured is an embedding of the graph corresponding to the map $m_\pi$ from Example \ref{ex:mapproduct}.  Here we have $\alpha = (4,3)$, $\beta = (3,2,2)$ and $\pi = (1)(2 \, 3 \, 5)(4 \,
	7 \, 6)$.  Crossings of edges are not true crossings as the embedding is on
	a surface of appropriate genus.  The embedding has two faces, so the graph is embedded on a surface
of genus $g = \tfrac{1}{2}(2 - 5+7-2) =1$.  The walk (in dashed purple) along a
face is described as a part of a cycle of $R \cdot E(\pi)$ in Example
\ref{ex:mapproduct}.}
\end{figure}

A second important fact can be observed from \eqref{eq:connection}:  the cycles
of $R\cdot E(\pi)$ naturally correspond to the cycles in $\sigma_0 \pi \omega_0
\pi^{-1}$.    To see why, consider the natural action of a permutation in
$\Sym_n$ acting on the symbols in $S$ and $T$ (the action of permuting indices).   It follows from the computation in \eqref{eq:connection} that if we
ignore the symbols in $R \cdot E(\pi)$ in $T$, the remaining permutation on the
symbols in $S$ is precisely the permutation $\sigma_0 \pi \omega_0 \pi^{-1}$
acting on the symbols in $S$ (our left to right convention for multiplying
permutations gives $(\sigma_0 \pi \omega_0 \pi^{-1}) (i)  = \pi^{-1}(\omega_0(\pi(
\sigma_0(i))))$).   Equivalently, the permutation $\sigma_0 \pi
\omega_0 \pi^{-1}$ acting on $S$ gives precisely half of the cycles of the permutation
$(R \cdot E(\pi))^2$;  the other cycles are given by the permutation $\omega_0
\pi^{-1} \sigma_0 \pi$ acting on the symbols in $T$.   This correspondence\footnote{The claim
	about the connection between the cycles of $\sigma_0 \pi \omega_0 \pi^{-1}$
	and $R \cdot E(\pi)$ remains true if we replace $\sigma_0$ and
	$\omega_0$ by any other
	fixed permutations $\sigma$ and $\omega$ in $C_\alpha$ and $C_\beta$,
	respectively, as long as the permutation $R$ has cycles compatible with
	$\sigma$ and $\omega$ in the sense of \eqref{eq:compatible}.} is
also
illustrated in Example \ref{ex:mapproduct}.  It follows that if $\sigma_0 \pi \omega_0 \pi^{-1}$ has cycle type $\lambda$, then $R \cdot E(\pi)$
has cycle type $2 \lambda := (2 \lambda_1, 2\lambda_2, \ldots)$, where the 
$\lambda_i$ are the parts of $\lambda$.   In particular,
\begin{equation}\label{eq:cyclessame}
	c(\sigma_0 \pi \omega_0 \pi^{-1}) = c(R \cdot E(\pi)).
\end{equation}

\begin{exampleblock}[Permutation/map correspondence]\label{ex:mapproduct}
	Let $\alpha = (4,3)$ and $\beta = (3,2,2)$.  A map $m_\pi = (D, R, E(\pi))$
	is pictured in Figure \ref{fig:mapproduct} with $\pi = (1)(2 \, 3 \, 5)(4 \,
	7 \, 6) \in \Sym_7$.  Here $D = S \cup T$ and
	\begin{align*}
		R &= (s_1 \, s_2 \, s_3\, s_4)(s_5 \, s_6 \, s_7) \cdot (t_1 \, t_2 \,
		t_3) (t_4 \, t_5) (t_6 \, t_7),\\
		E(\pi) &= (s_1 \, t_1)(s_2 \, t_3)(s_3 \, t_5)(s_4 \, t_7)(s_5 \,
		t_2)(s_6 \, t_4)(s_7 \, t_6).\\
	\end{align*}
	Whence
	\begin{equation}\label{eq:rdote}
		R \cdot E(\pi) = (s_1\, t_3)(s_2 \, t_5 \, s_6 \, t_6 \, s_4 \, t_1\,
		s_5 \, t_4\, s_3 \, t_7 \, s_7 \, t_2).
	\end{equation}
	and
	\begin{equation}\label{eq:resquared}
		(R \cdot E(\pi))^2 = (s_1)(t_3)(s_2\, s_6 \, s_4\, s_5\, s_3 \, s_7)(t_5\, t_6\, t_1\, t_4 \, t_7 \, t_2).
	\end{equation}

	Each cycle of $R \cdot E(\pi)$ traces a face of $m_\pi$ in Figure
	\ref{fig:mapproduct} via a walk in the following way.  Pick a symbol, say $t_5$, and
	begin a walk around the map while keeping edges to the right;  begin
	at the vertex containing the dart $t_5$, which is sent to $t_4$ via $R$;
	then $E(\pi)$ sends $t_4$ to $s_6$.
	Thus $R \cdot E(\pi)$ sends $t_5$ to $s_6$, and this is indicated by the
	dashed purple walk.  Continue, via $R$, from $s_6$ to $s_7$;  then $E(\pi)$
	sends $s_7$ to $t_6$.  Thus $R \cdot E(\pi)$, and thus the walk,  sends 
	$s_6$ is $t_6$.  Hence the walk
	traces $\cdots t_5 \rightarrow s_6 \rightarrow t_6 \cdots$.  If one
	continued in this fashion, one would follow the symbols of the larger cycle
	in $R \cdot E(\pi)$ given in \eqref{eq:rdote} along the face indicated by
	the dashed purple walk until $t_5$ is revisited.

	To see the correspondence between $R \cdot E(\pi)$ and the permutations
	$\sigma_0 \pi \omega_0 \pi^{-1}$ and $\omega_0 \pi^{-1} \sigma_0 \pi$, we
	find the canonical permutations $\sigma_0$ and $\omega_0$ of type $\alpha$ and
	$\beta$, respectively, are given by
	\begin{align*}
		\sigma_0 &= (1 \, 2 \, 3\, 4)(5 \, 6 \, 7) \in C_{\alpha},\\
		\omega_0 &= (1 \, 2\, 3)(4 \, 5)(6 \, 7) \in C_{\beta}
	\end{align*}
    We find 
	\begin{align*}
		\sigma_0 \pi \omega_0 \pi^{-1} &= (1 \, 2 \, 3\, 4)(5 \, 6 \, 7) \cdot (1)(2 \, 3 \, 5)(4 \, 7 \, 6) \cdot (1
	\, 2 \, 3)(4 \, 5)(6 \, 7) \cdot [(1)(2 \, 3 \, 5)(4 \, 7 \, 6)]^{-1}\\
				&= (1)(2 \, 6 \, 4\, 5\, 3\, 7)
	\end{align*}
	and
	\begin{equation*}
		\omega_0 \pi^{-1} \sigma_0 \pi = (3)(1\, 4\, 7\, 2\, 5\, 6),
	\end{equation*}
	which the reader is invited to compare to $(R \cdot E(\pi))^2$ in \eqref{eq:resquared}.\qed
\end{exampleblock}

We denote by $M_{\alpha, \beta}$ the set of all such maps:
    $$
        M_{\alpha, \beta} := \{ m_\pi : \pi \in S_n \}
    $$
    Let $U$ be the uniform probability distribution on this set. Then
    $(M_{\alpha, \beta}, U)$ is a probability space.
We use $F(\cdot)$ to denote the number of faces of a map.  From
\eqref{eq:facescycles} and \eqref{eq:cyclessame}, we see that
$F(m_\pi) = c(R \cdot E(\pi))= c(\sigma_0 \pi \omega_0 \pi^{-1})$. Define $F_{\alpha,\beta} : = F$ as the random variable on
 $(M_{\alpha,\beta},U)$ for the number of faces in a map in
 $M_{\alpha,\beta}$.  We now show how $F_{\alpha, \beta}$ is related $C_{\alpha,
 \beta}$.

\begin{lemma} \label{lem:processequiv}
	For any positive integer $n$ and partitions $\alpha$ and $\beta$ of $n$,
    $$\E[F_{\alpha,\beta}] =\E[C_{\alpha,\beta}].$$
\end{lemma}
\begin{proof}
	Note that
\begin{equation*}
	 \E[F_{\alpha,\beta}]= \deriv\, \frac{1}{n!} \sum_{\pi \in S_n} q^{c(R \cdot
	 E(\pi))},
\end{equation*}
while
\begin{equation*}
	\E[C_{\alpha,\beta}] = \deriv\, \frac{1}{\vert C_\alpha \vert \vert C_\beta
	\vert } \sum_{\substack{\sigma \in C_\alpha \\ \omega \in C_{\beta} }}
	q^{c(\sigma \omega)}.
\end{equation*}
With $\sigma_0$ and $\omega_0$ as the canonical permutations of type $\alpha$ and
$\beta$, respectively, we have 
    \begin{align*}
       \frac{1}{n!} \sum_{m_\pi \in M_{\alpha, \beta}} q^{F(m_\pi)} &= \frac{1}{n!} \sum_{\pi \in S_n} q^{c(R \cdot E(\pi))} \\
        &= \frac{1}{n!}\sum_{\pi \in S_n} q^{c(\sigma_0 \pi \omega_0
		\pi^{-1})}\,\,\,
	(\textnormal{from \eqref{eq:cyclessame}})\\
        &= \frac{z_\beta}{n!}  \sum_{\omega \in C_{\beta}} q^{c(\sigma_0 \omega)} \\
        &= \frac{1}{\vert C_\alpha \vert \vert C_\beta \vert}
		\sum_{\substack{\sigma \in C_\alpha \\ \omega \in C_{\beta} }}
		q^{c(\sigma \omega)},
    \end{align*}
	and the result follows.
\end{proof}
Thus studying the number of faces of bipartite maps -- the embedded
graphs associated with maps $(D, R, E(\pi))$ as $\pi$ ranges in $\Sym_n$
-- is equivalent to studying $C_{\alpha,\beta}$.  We use this correspondence to
estimate $E[C_{\alpha,\beta}]$ because it is easier to define our structures and
substructures on maps, specifically on their associated embedded graphs. 


The statistic $C_{\alpha,\beta}$ has been studied previously in some special
cases.  The case when $\alpha = 2^{n/2}$, the partition with all parts equal to
2, models a random surface obtained by gluing together polygonal disks and has
been studied with motivation from physics.  Pippenger and Schleich
\cite{pippenger2006topological} show that $\E[C_{2^{n/2},3^{n/3}}] = H_n +
O(1)$, and also obtain bounds for the variance.  Their results use a
combinatorial analysis of a random process.  More recently, their result was
generalised by Chmutov and Pittel \cite{chmutov2016surface} to any $\beta$ with
parts all of size at least $3$.  They show that the whole distribution of
permutations obtained is asymptotically uniform up to parity.  As a corollary to
their result, it follows that $\E[C_{2^{n/2},\beta}] = H_n + O(1)$ when $\beta$
has parts all of size at least $3$.  Their proof follows a
method of Gamburd \cite{gamburd2006poisson} that uses
the Fourier transform on representations of the symmetric group and recent
bounds on characters given by Larsen and Shalev \cite{larsen2008characters}.

Although it is not explicitly mentioned in their paper, these methods trivially
extend to the case when $\alpha$ has all parts of size at least $2$ and $\beta$
has all parts of size at least $3$.  The same result holds in this case, giving $\E[C_{\alpha,\beta}] = H_n + O(1)$.  The size of the
$O(1)$ term is unclear, as it relies on a very general, powerful character
bound.  However in some special cases more precise bounds are known.  The case
when $\alpha = (n)$ has been studied by various authors, using representation
theory \cite{stanley2011two}, matrix integrals \cite{Ja94}, and more
combinatorial methods Goulden and Jackson \cite{goulden1992combinatorial}, Cori
\emph{et al.} \cite{cori2012odd}, and Boccara \cite{boccara1980nombre}.  These
results are more refined, and give the exact number of products with a given
number of cycles.  From this, we can obtain the average number of cycles.  The
case $\alpha = \beta = (n)$ is especially nice; we have that $\E[C_{(n),(n)}] =
H_{n-1} + \lceil \frac{n}{2}\rceil ^{-1} $.  In the more general case when
$\beta$ has all parts of size at least $2$, it follows from Stanley's generating
function \cite[Theorem 3.1]{stanley2011two} that $H_{n} - \frac{5}{n} \leq \E[C_{(n),\beta}] \leq H_{n} + \frac{3}{n}$   (see \cite{loth2021random}).

In all these known cases, we see that the average number of cycles is close to
$H_n$.  We shall here be determining an estimate in the more general case when
$\alpha$ and $\beta$ can be any partitions with no parts of size one.  We state
our main theorem.
\begin{theorem} \label{thm:mainbound}
	Let $n$ be a positive integer.  For any partitions $\alpha$ and $\beta$ of
	$n$ with no part of size 1, we have $H_n - 3 < \E[C_{\alpha, \beta}] \leq H_n + 1$.
\end{theorem}
Our theorem eliminates the asymptotic notation present in the more general
results previously mentioned.  The method of proof will also be more
combinatorial, using random processes on maps.

\section{Basic terminology and random processes}
We give some definitions and terminology that we shall make extensive use of throughout the paper. 

Fix $\alpha$ and $\beta$ to be partitions of $n$ with no parts of size 1.  A partial map is a triple $m =
(D,R,E)$, where $D$ and $R$ are defined as they were for maps, but where $E \in
\Sym_{D}$ is
    an involution (not necessarily fixed point free).  The involution can be
	represented by an injection $\pi : X \rightarrow [n]$, where
	$X \subseteq [n]$.  Then we can set $$E(\pi) := \prod_{i \in X} (s_i \, t_{\pi(i)}) $$
    and define the partial map $m_\pi= (D,R,E(\pi))$. 
Notice that a map is itself a partial map.  In this case $E$ will be a fixed
point free involution, and the associated function $\pi$ will be a permutation on
$\{1,2, \ldots, n\}$.  Also when $E$ is the identity (the associated function $\pi$
has the empty set as its domain) the partial map is just a set
of darts at vertices with no faces or edges.  These are the two extreme cases.

Fix an injection $\pi : X \rightarrow [n]$ for some subset $X \subseteq [n]$,
and let $m=m_\pi =(D,R,E(\pi))$ be a partial map, and write $E = E(\pi)$.
Figure \ref{fig:firstuexample} contains some examples of the definitions given
below.
    \begin{itemize}
        \item \emph{Paired/unpaired darts}: Each dart in $m$ is one of the symbols in $S$ or $T$.  We have two types of dart:
        \begin{itemize}
            \item Paired dart: A dart in a partial map that is part of an edge,
		    \emph{i.e} a dart that is in a $2$-cycle in $E$. 
            \item Unpaired dart: A dart in a partial map that is not part of an
		    edge, \emph{i.e} a dart that is a fixed point in $E$.  Let $S^u, T^u$ be the set of unpaired darts in $S,T$ respectively.
        \end{itemize}
        \item \emph{Completed face}: A cycle in $R \cdot E$ only containing
			paired darts.  Let $F(m)$ denote the number of completed faces
			in the partial map.\
        \item \emph{Unpaired permutation/partial face/length}: For 
			$m$, the induced permutation of $R
			\cdot E$ on $S^u \cup T^u$ is called the \emph{unpaired
			permutation}, which we denote $u_m$.  That is, we
		remove all paired darts, and any resulting empty cycles, from $R \cdot
		E$ to obtain $u_m$.  We write $u=u_m$ when $m$ is clear from the
		context.   A \emph{partial face} is a cycle in $u_m$.  The \emph{length}
		of a partial face is its length as a cycle.
	\item \emph{Bad dart}:  An unpaired dart contained in a cycle of length one in $u_m$.   
		\item \emph{Mixed partial face/bad partial map}:  A partial face in 
			$m$ is \emph{mixed} if it contains darts in both $S^u$
			and $T^u$.  The partial map $m$ is called \emph{bad} if it does not
			contain any mixed partial faces.
	\end{itemize}

\begin{figure}
    \centering
   \includegraphics[scale=0.8]{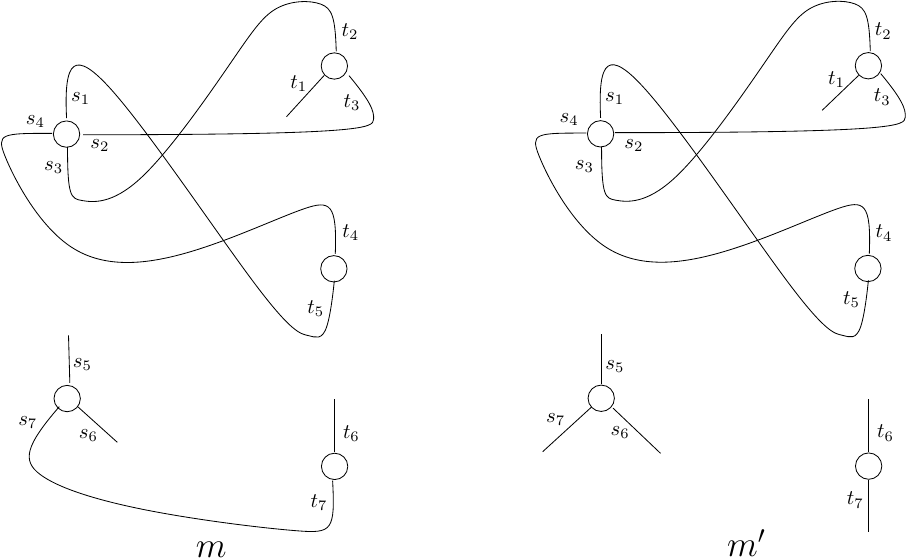} \label{fig:firstuexample}
\caption{ In the partial map $m$ on the left, the paired darts are $\{ s_1, s_2,
	s_3, s_4, s_7, t_2, t_3, t_4, t_5, t_7 \}$.  The unpaired darts are all
	of the remaining darts,  so $S^u = \{s_5, s_6 \}$ and $T^u =
	\{t_1, t_6 \}$.  It has two completed faces, given by the cycles
	$(s_2\, t_2)$ and $(s_4\, t_5)$.  The unpaired permutation is
	$u_m = (t_1)(s_5 \, s_6\, t_6)$.  It has one
	bad dart, $t_1$, and it has one mixed partial face, so it is not a bad map.  The partial map $m'$ on the right has $S^u = \{s_5, s_6, s_7\}$ and
	$T^u = \{t_1, t_6, t_7\}$, and it has the same bad dart, $t_1$, as $m$.  Furthermore, we observe $u_{m'} = (t_1)
	(s_5\,s_6 \, s_7 \, s_8)(t_6 \, t_7)$.  There are no mixed partial faces, so
	$m'$ is a bad map.
}
\end{figure}

We shall use a different random process to prove the upper and lower bounds for $\E[F_{\alpha, \beta}]$.
\\

\noindent \textbf{Random Process A (RPA)}

\begin{enumerate}
    \item Initialize $S^u=S$, $T^u = T$, $X=\emptyset$, $\pi_0: X \rightarrow [n]$.
    \item Call $m_{\pi_{k-1}}$ the partial map at the start of the $k^\upth$
	    step.  Pick the \emph{active dart} $d$ from $S^u \cup T^u$ with respect to the following order of preference.
    \begin{itemize}
        \item The bad dart in $S^u$ with smallest index.
        \item The bad dart in $T^u$ with smallest index.
        \item The dart with the smallest index in $S^u$.
    \end{itemize}
    \item If $d = s_i \in S^u$, then pick $t_j \in T^u$ uniformly at random and
		call $t_j$ the \emph{pairing dart} at this step.  If $d = t_j \in T^u$,
		then pick $s_i \in S^u$ uniformly at random and call $s_i$ the
		\emph{pairing dart}.  Set $X = X \cup \{i\}$, and define $\pi_k: X
		\rightarrow [n]$ as $\pi_{k-1}$ from the previous step, but also with $\pi_k(i) = j$.
    \item Then $S^u = S^u - \{s_i\}$ and $T^u = T^u - \{t_j\}$.  If $S^u \neq \emptyset$ then return to step $2$.
    \item Output the map $m_{\pi_n}$.
\end{enumerate}

\noindent \textbf{Random Process B (RPB)}

\begin{enumerate}
    \item Initialize $S^u=S$ and $T^u = T$, $X=\emptyset$, $\pi_0: X \rightarrow [n]$.
    \item Call $m_{\pi_{k-1}}$ the partial map at the start of the $k^\upth$ step.  Pick the \emph{active dart} as $s_k$.
    \item Pick $t_j \in T^u$ uniformly at random and call $t_j$ the \emph{pairing dart} at this step.  Set $X = X \cup \{i\}$, and define $\pi_k: X \rightarrow [n]$ as $\pi_{k-1}$ from the previous step also with $\pi_k(i) = j$.
    \item Then $S^u = S^u - \{s_k\}$ and $T^u = T^u - \{t_j\}$.  If $S^u \neq \emptyset$ then return to step $2$.
	\item Output the map $m_{\pi_n}$.
\end{enumerate}
After $n$ steps of either process, the output is a map in $M_{\alpha,\beta}$.

\begin{lemma}
    RPA and RPB both output a uniform at random map from $M_{\alpha,\beta}$.
\end{lemma}
\begin{proof}
    At the $k^\upth$ iteration of step $2$, there are $n-k+1$ choices of pairing
    dart.  There are therefore $n!$ possible outcomes of this random process,
    each with equal probability.  There are $n!$ elements in $M_{\alpha,\beta}$,
    and every map is clearly output by each of the random processes.  The result follows.
\end{proof}

We give some basic observations, which are illustrated in  Figures
\ref{fig:mapobservationsexample} and \ref{fig:mapobservationsexample2}. 

\begin{observation}\label{thm:obs}
Suppose that we are at some step 2 of RPA or RPB.  Let $m$ be the
partial map at the beginning of this step, and suppose that it has active dart
$d$.  Step 3 then pairs $d$ with another dart.  Let $u=u_m$.  Observe the following.
    \begin{enumerate}
	    \item[1.] Suppose that $d$ is a bad dart.  A completed face is added at this step
			if and only if the pairing dart is also bad (\emph{i.e.}  in the
		unpaired permutation $u$, both $d$ and its pairing dart are
		fixed points).  When a completed face is added the resulting map has two
		fewer bad darts.  A new bad dart
		is added if and only if the pairing dart is in a partial face of length $2$.
		\item[2.] Suppose that $d$ is not a bad dart.   A completed face is added if and
		only if the pairing dart is $u(d)$ or $u^{-1}(d)$.  If the pairing dart is $u(d)$ and $u(d) = u^{-1}(d)$ then two completed faces are added at this
		step.  Similarly, a bad dart is added if and only if the pairing
		dart is $u^2(d)$ or $u^{-2}(d)$.  If $u^2(d) = u^{-2}(d)$, then
		two bad darts are added.
    \end{enumerate}
\end{observation}
\begin{figure}
    \centering
    \includegraphics[scale=0.8]{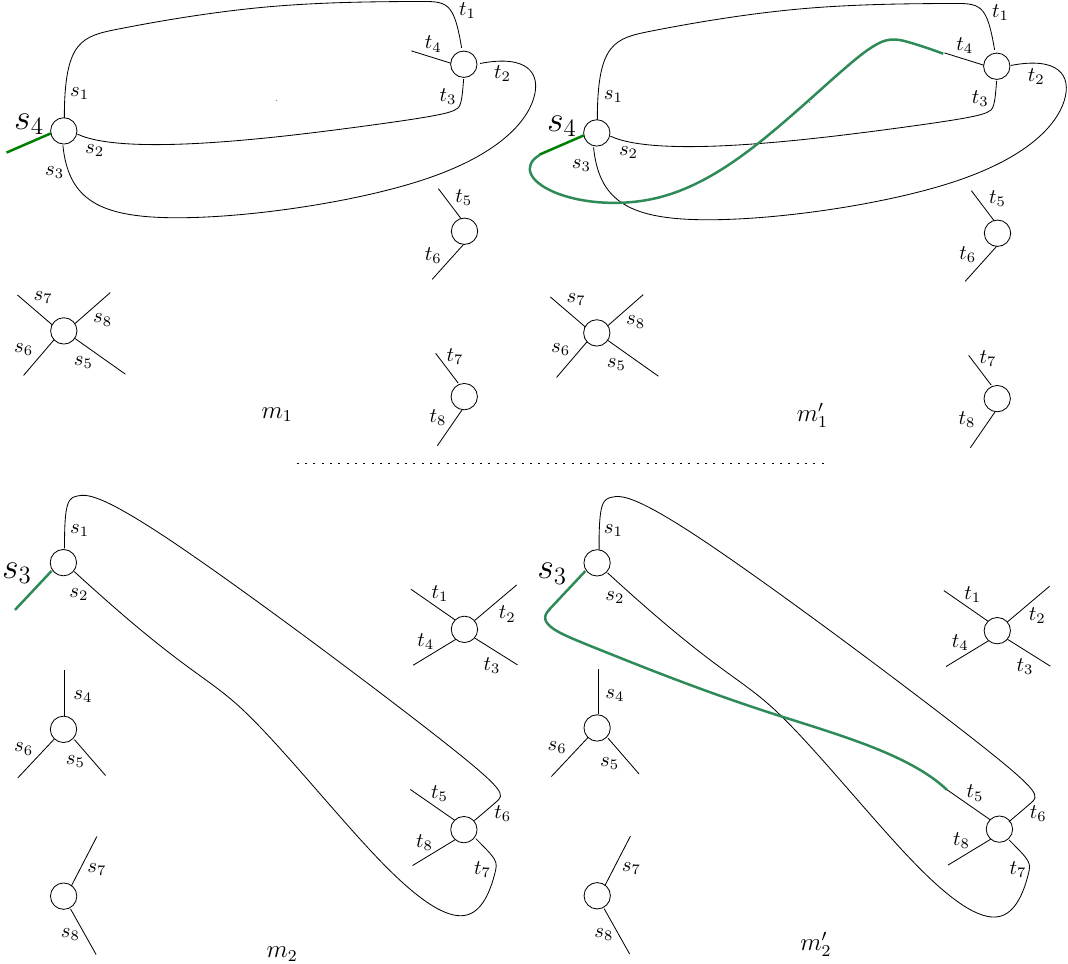}
    \caption[Applying RPAB \ref{ran:generalconjugacyupper} to a partial map.]{
	    Illustrating Observation \ref{thm:obs}.1, we have four maps:  the left
		contains two partial maps $m_1$ and $m_2$ and the right contains possible outputs $m_1'$ and $m_2'$, respectively, of an application of one of the random processes.  On
		the left the active dart is in green.  The map $m_1$ has two bad darts $s_4$
		(the active dart) and $t_4$.  The map $m_1'$ was obtained by pairing $s_4$
		with $t_4$, completing a face (given by the cycle $(s_1\, t_3\, s_4\,
		t_1\, s_3\, t_4)$), as claimed in Observation 5.1.\newline
		The map $m_2$ has $s_3$ as the active dart, and has a face of length 2.
		If $s_3$ is paired with $t_5$, a bad dart, $t_8$, is created in $m_2'$.
}
    \label{fig:mapobservationsexample}
\end{figure}
\begin{figure}
    \centering
    \includegraphics[scale=0.8]{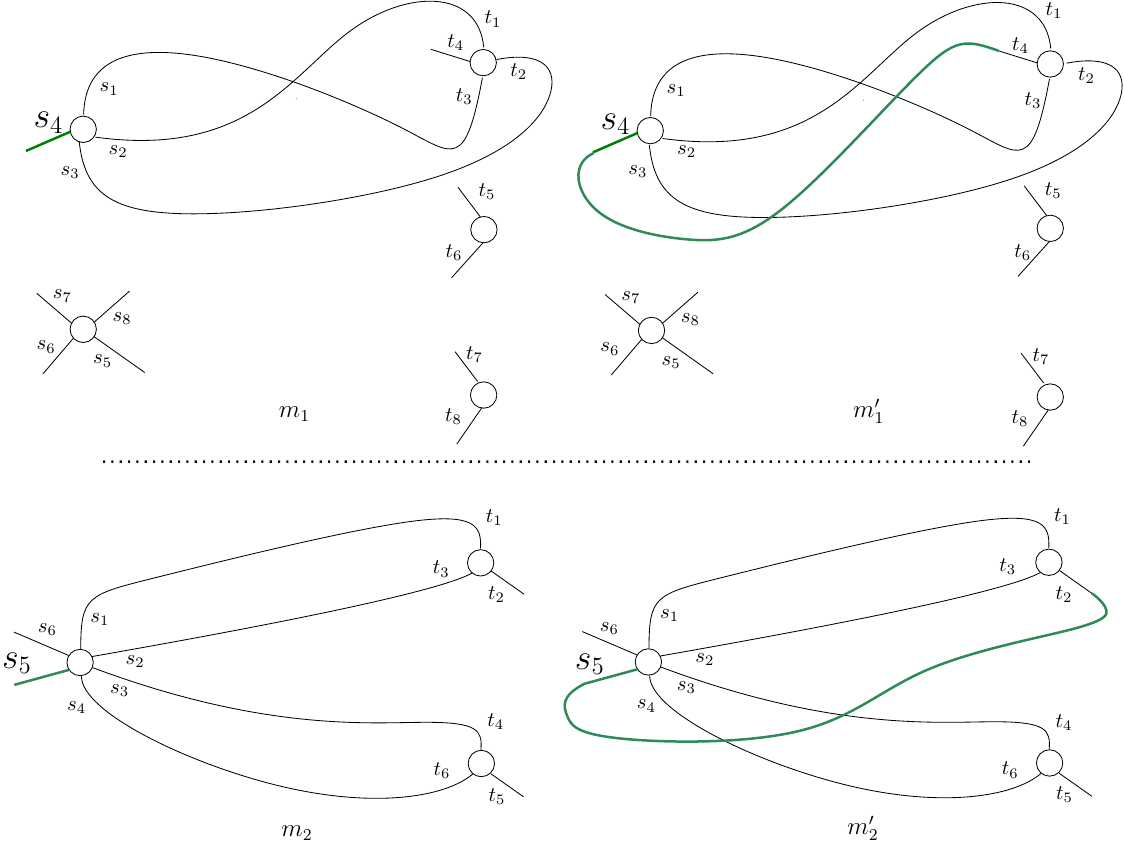}
    \caption[lof entry]{Illustrating Observation \ref{thm:obs}.2, we have four
		maps whose relationship with each other is the same as in Figure
		\ref{fig:mapobservationsexample}.  The green dart is again the active
		dart.  We find $u_{m_1} = (s_5\, s_6\, s_7\, s_8)(t_5\, t_6)(t_7\,
		t_8)(s_4\, t_4)$.  The
    partial map $m_1'$ is obtained by pairing $s_4$ with $t_4$;  by Observation
    \ref{thm:obs}.2 this creates two completed cycles, one is given by $(s_1\,
    t_1\, s_3\, t_4\, s_2\, t_2)$ and the other by $(t_3\, s_4)$.\newline
    The map $m_2$ has $u_{m_2} = (s_5\, s_6\, t_2\, t_5)$.  The partial map
    $m_2'$ is obtained from $m_2$ by pairing the active dart $s_5$ with $t_2$.
    Since $u^2_{m_2}(s_5) = u^{-2}_{m_2}(s_5) = t_2$, this creates two bad darts in
$m_2'$, the darts $s_6$ and $t_5$.}
    \label{fig:mapobservationsexample2}
\end{figure}

\subsection{Proving the upper bound with RPA}\label{sec:rpa}

Again, throughout this section $\alpha$ and $\beta$ are partitions of some
positive integer $n$ with no parts of size 1.
Suppose some partial map $m$ appears at the start of some step of RPA.  Then the random process will pair the active dart with a randomly chosen
pairing dart.  This will add some number (possibly zero) of completed faces to
the partial map.  Let $F_m$ be the random variable for the number of completed
faces that are added when the random process adds a new edge to $m$.

\begin{lemma} \label{lem:akdecomp}
    Suppose the random process outputs the map $m_\pi$.   Let $m_\pi^{(k)}$
	denote the partial map at the $k^\upth$ step of the process
that outputted $m_\pi$.  
Then we have
    $$
        \E[F_{\alpha,\beta}] \leq \max_{m_\pi \in
		M_{\alpha,\beta}}\left\{\sum_{k=1}^n \E[F_{m_\pi^{(k-1)}}]\right\}.
    $$
\end{lemma}
\begin{proof}
    Each face in $m_\pi$ is added as a completed face at some step of RPA.  Once it is added, it cannot be removed at a later step.  At the
	$k^\upth$ iteration of step 3, let $m_k$ be the random variable for the
	partial map at the start of this step, and let $A_k$ be the random variable
for the number of completed faces added at the $k^\upth$ step.
    
    Using total probability, we have
    \begin{equation*}
        \E[A_k] = \sum_m Pr[m_k = m] \E[F_m],
    \end{equation*}
    where the sum is over all possible partial maps.
    Using linearity of expectation gives:
    \begin{align*}
        \E[F_{\alpha,\beta}] &= \sum_{k=1}^n \E[A_k] \\
        &= \sum_{k=1}^n \sum_m Pr[m_k = m] \E[F_m] 
	\end{align*}
	\begin{align*}
        &= \frac{1}{n!} \sum_{m_\pi \in M_{\alpha,\beta}} \sum_{k=1}^n \E[F_{m_\pi^{(k-1)}}] \\
        &\leq \max_{m_\pi \in M_{\alpha,\beta}}\left\{\sum_{k=1}^n \E[F_{m_\pi^{(k-1)}}] \right\} \qedhere
    \end{align*}
\end{proof}

\begin{lemma} \label{lem:baddartsactivefaces}
    Any partial map $m$ appearing at the start of some step of RPA has at most two bad darts, and at most one mixed partial face.  This mixed partial face is always of the form $(s_{i_1} \, \dots \, s_{i_a} \, t_{j_1} \, \dots \, t_{j_b})$, and if the active dart $d$ is in this partial face it must equal $s_{i_1}$.
\end{lemma}
\begin{proof}
    Suppose the partial map $m$ has no bad darts, and suppose $d$ is the active
	dart at this step.  If the pairing dart is $u^2(d)$ then $u(d)$ will become
	a bad dart at the next step.  If the pairing dart is $u^{-2}(d)$ then
	$u^{-1}(d)$ will become a bad dart at the next step.  If $u^2(d) =
	u^{-2}(d)$ then these choices coincide and we add two bad darts.  The above
	are merely Observation \ref{thm:obs} restated.  In either case, we add at most two bad darts.

        Now suppose $m$ has one or two bad darts.  Then $d$ is one of these bad
		darts, since step 2 of RPA prioritises choosing bad darts for the active
		dart.   See Figure \ref{fig:mapobservationsexample}
		for an example of RPA processing a bad dart.  We make a new bad dart if
		and only if the pairing dart $d'$ is in partial face of length 2.  In
		the new map $d$ will no longer be a bad dart.  Since we removed a bad dart, and added at most one bad dart, at the next step we will have at most two bad darts.

        This covers both cases, so the number of bad darts is always at most
		two.  We now prove the claims about the mixed partial faces.

        The partial map at the start of an application of RPA has no
		paired darts, so it has no mixed partial faces.  We therefore proceed by
		induction.  Let $m$ be the partial map at the beginning of some
		iteration of RPA.  Suppose the partial map $m$
		at the start of a step of the process has at most one mixed partial
		face.  If the map is bad then it has no mixed partial faces, and we can
		add at most one mixed partial face at this step.   

		Otherwise assume that the map has one mixed partial face.  We have two cases:
		the active dart $d$ is bad or not.  If $d$ is bad, then processing it
		cannot add any more mixed partial faces.  If $d$ is not bad, then it is
		in $S^u$ by our choice of active dart in step 2.  Let $v$ be the vertex
		incident with $d$.  We first aim to show that $u_m^{-1}(d) \in T^u$;
		for that purpose, we assume the opposite: that $u_m^{-1}(d) \in S^u$.  First, all the
		darts in $S^u$ with smaller index than $d$ are processed.  Second,
		observe that all the darts in $S^u$ not incident with $v$ with larger index must be in
		partial faces only containing darts in $S^u$.  This is because any
		paired dart with larger index than $d$ in $S^u$ (at $v$ or a later
		vertex) must have been paired with an active bad dart in $T$ (by
		the choice of active dart at step 2), and
		pairing an active bad dart cannot create a new mixed partial
		face.  This has two consequences.  First, all the darts
		in $S^u$ that are contained in the mixed partial face must be
		incident with $v$.  Moreover, they appear contiguously in the
		mixed partial face from the dart with lowest index to the dart
		with highest index.  Since $d$ is the dart with lowest index
		among the unpaired darts at $v$, if $u_m^{-1}(d) \in S^u$, then
		$u_m^{-1}(d)$ has the highest index amongst unpaired darts at
		$v$.  Therefore the mixed partial face equals $(d \cdots
		u_m^{-1}(d))$, so only contains darts incident with $v$.  Thus
		this partial face is not mixed, giving a contradiction. 
        
        Therefore $u_m^{-1}(d) \in T^u$, so the mixed partial face is of the
	form\newline $(d \, s_{i_1} \, \cdots \, s_{i_a} \, t_{j_1} \, \cdots \, t_{j_b})$
	with $a \geq 0$ and $b \geq 1$.  Pairing $d$ with any dart in this mixed partial face, or outside of this mixed partial face, still preserves there being at most one mixed partial face of this form.
\end{proof}

We now estimate when completed faces are added to a partial map.

\begin{proof} [Proof of the upper bound in Theorem \ref{thm:mainbound}]
    Fix some map $m_\pi \in M_{\alpha,\beta}$.  Using Lemma \ref{lem:akdecomp}
    it will be sufficient to estimate $\sum_{k=1}^n \E[F_{m_\pi^{(k-1)}}]$.  Fix
	some value of $k$ and let $d$ be the active dart chosen in the $k^\upth$ iteration of step 2.
    
	Case 1: The dart $d$ is bad.  By Observation \ref{thm:obs}, a face will be added if and only if the
    pairing dart is also bad.  By Lemma \ref{lem:baddartsactivefaces}, there is
    at most one other bad dart, so there is at most one other choice that adds
    a face.  This gives $\E[F_{m_\pi^{(k-1)}}] \leq \frac{1}{n-k+1}$.
    
    Case 2: The dart $d$ is not bad and $u(d) \in S^u$.   Then the pairing dart cannot be
    chosen to be $u(d)$, so $u^{-1}(d)$ is the only possible choice of pairing
    dart that adds a completed face.  In this case $\E[F_{m_\pi^{(k-1)}}] \leq \frac{1}{n-k+1}$.
    
    Case 3: The dart $d$ is not a bad dart and $u(d) \in T^u$.  Then choosing
	the pairing dart to be $u(d)$ or $u^{-1}(d)$ adds a completed face, and both
	of these choices are possible.  If $u(d) = u^{-1}(d)$ then choosing the
	pairing dart to be $u(d)$ adds $2$ completed faces.  We therefore have
	$\E[F_{m_\pi^{(k-1)}}] \leq \frac{2}{n-k+1}$.  Since $u(d) \in T^u$, we see
	that $d$ is in a mixed partial face at the start of this step.  By Lemma \ref{lem:baddartsactivefaces} we have that $u^{-1}(d) \in T^u$, so the mixed partial face is of the form $(d \, t_{i_1}\, \dots \, t_{i_a})$.  Therefore once $d$ is paired, this face is no longer mixed.  By Lemma \ref{lem:baddartsactivefaces} this was the only mixed partial face at the start of this step, so the map at the start of the next step is bad.
    
    Let $d'$ be the active dart at the next step, which is step $k+1$.  If the
	active dart $d'$ is bad, then we are in case 1 and we may add more faces.
	In this case, the partial map at the start of step $k+2$ will also be bad.
	This will continue until we eventually arrive at some step $j$, with $j >
	k$, where the partial map at the start of this step is bad and the active
	dart is not bad, or the whole map will be completed before we reach such a
	step $j$. If this step $j$ exists, and $\hat{d}$ is the active dart, then because the map is bad, we have
	$u_{m_\pi^{(j-1)}}(\hat{d}) \in S^u$ and $u^{-1}_{m_\pi^{(j-1)}}(\hat{d})
	\in S^u$, so there will be no choice of
	pairing dart that adds a face.   Therefore $\E[F_{m_\pi^{(j-1)}}] = 0$.
	Grouping together these terms we obtain $\E[F_{m_\pi^{(k-1)}}] +
	\E[F_{m_\pi^{(j-1)}}] \leq \tfrac{2}{n-k+1} < \tfrac{1}{n-k+1} +
	\tfrac{1}{n-j+1}$, with the intermediate terms $\E[F_{m_\pi^{(i-1)}}] \leq
	\tfrac{1}{n-i+1}$.  Thus
	\begin{equation}\label{eq:case3run}
		\sum_{i=k}^j \E[F_{m_\pi^{(i-1)}}] \leq \sum_{i=k}^j \frac{1}{n-i+1}.
	\end{equation}
	If the index $j$ exists, we call the above described run between $k$ and $j$
	of the algorithm a \emph{closed run}, and a \emph{open run} otherwise.  The
	algorithm may have several closed runs, and potentially one open run.  The
	expectations for added faces for closed runs all satisfy
	\eqref{eq:case3run}.  Then, by definition, all remaining iterations of the
	algorithm satisfy Cases 1 or 2, in which case
	\begin{equation*}
		\sum_{k=1}^n \E[F_{m_\pi^{(k-1)}}] \leq \sum_{k=1}^{n} \frac{1}{n-k+1}
		= H_n < H_n+1.
	\end{equation*}
	If there is an open run beginning at the $k^\upth$ iteration, then
	\begin{equation*}	
		\sum_{i=k}^n \E[F_{m_\pi^{(i-1)}}] \leq \frac{2}{n-k+1} + \sum_{i=k+1}^n
		\frac{1}{n-i+1} \leq \left(\sum_{i=k}^n \frac{1}{n-i+1}\right) + 1,
	\end{equation*}	
	in which case we also obtain 
	\begin{equation*}
		\sum_{k=1}^n \E[F_{m_\pi^{(k-1)}}] \leq \left(\sum_{k=1}^n
			\frac{1}{n-k+1} \right) + 1 \leq H_n + 1.
	\end{equation*}	
	In either case, the result is proved.
\end{proof}

\subsection{Proving the lower bound with RPB}
Again, throughout this section $\alpha = (\alpha_1, \ldots)$ and $\beta =
(\beta_1, \ldots)$ are partitions of some positive integer $n$ with no parts of
size 1.  Recall that for each $0 \leq j \leq \ell(\alpha)$
that $\alpha'_j = \sum_{i=1}^j \alpha_i$, with $\alpha'_0 = 0$.

We prove the lower bound by analysing RPB.
In this case the active dart at the beginning of step $k$ is always $s_k$.  Let
$B_k$ be the random variable for the number of completed faces added to the
partial map at step $k$.  Thus $B_k$ is similar to $A_k$ except it applies to
RPB.  As for $A_k$, we clearly have
\begin{equation}  \label{lem:bkdecomp}
	\E[F_{\alpha,\beta}] = \sum_{k=1}^n \E[B_k].
\end{equation}

We will need some additional variables on RPB.  Let $m$ be the partial map
at the start of step $k$ of RPB.  Note that $m$ has $k-1$ edges and
$2(n-k+1)$ unpaired darts.
    \begin{itemize}
        \item Let $O_k$ be the random variable for the number of bad darts in $T^u$ in $m$.
        \item Let $b_k$ be the event that $m$ is bad.
    \end{itemize}
Thus each variable $O_k$ and $b_k$ count their relevant quantities at the
beginning of step $k$ of the algorithm.
Before we find estimates on these new definitions, we make an observation about
RPB.  These observations are based on the simple choice of active dart in step
2.
\begin{observation}\label{thm:obs2}
		Let $m$ be a map at the beginning of step $k$.  Suppose the active dart
		$s_k$ is at the $j^\upth$ vertex, so  $k=\alpha_j' + r$ for some $0 \leq r \leq
		\alpha_{j+1}-1$.  Then all the darts at the $j^\upth$ vertex after $s_k$
		are unpaired, while all the darts before are paired.   Furthermore, all
		the darts at the $i^\upth$ vertex on the
		left hand side are paired if $i < j$ and unpaired if $i > j$.  It follows that a cycle of
		$u_m$ is $(s_k \cdots s_{\alpha'_{j+1}}\, t_{i_1} \cdots t_{i_p})$ for
		some set of darts $M=\{t_{i_1}, \ldots, t_{i_p}\} \subseteq T$.  This is the only potentially
		mixed partial face of $m$, and $m$ is bad if and only if $M$ is empty.  In the special case where $r =
		1$, the set $M$ is necessarily empty, and the above cycle of $u_m$ has the form
		$(s_{\alpha'_j+1} \cdots s_{\alpha'_{j+1}})$.
\end{observation}

We now give estimates on $O_k$ and $b_k$.

\begin{lemma} \label{lem:baddartsestimate}
    If $k \neq \alpha'_j + 1$ for all $j$, then $\E[O_k] \leq 3$ and $Pr[b_k] \leq \frac{4}{n-k+2}$.
If $k = \alpha'_j + 1$ for some $j$ or $k=1$, then $\E[O_k] \leq 3 + \frac{3}{n-k+2}$ and $Pr[b_k] = 1$.
\end{lemma}
\begin{proof}
    Our proof is by induction on $k$, the steps of the algorithm.  Since $\alpha$ and
    $\beta$ have no parts of size one, we have $\E[O_1]  = \E[O_2] = 0$, $P[b_1]
    = 1$ and $Pr[b_2] = 0$.  Now suppose the conditions of the lemma hold for
	$O_j$ and $b_j$ for all $j \leq k$.  We prove the result holds after the
	$k^\upth$ step is completed (the beginning of the $(k+1)^\upst$ step).  Denote the
	map at the beginning of the $k^\upth$ step $m$; the active dart in $m$ is
	$s_k$ and there are $n-k+1$ unpaired darts in $T^u$ that $s_k$ could be
	paired with at this step.  We set $u=u_m$ to be the unpaired permutation.

Case 1: $k \neq \alpha'_j$ and $k \neq \alpha'_j + 1$ for all $j$.
From Observation \ref{thm:obs},  we add a bad dart at this step if and only if the pairing dart is $u^2(s_k)$ or
$u^{-2}(s_k)$.  However, by our choice of $k$, we have $u(s_k) \in S^u$, so 
pairing $s_k$ with $u^2(s_k)$ would add a bad dart in $S^u$ in the resulting map.  Therefore there is only
one choice of pairing dart for $s_k$ that could add a bad dart in $T^u$.
Separately, if the pairing dart is a bad dart, then the pairing dart is no longer bad in the
resulting partial map.  Since $k \neq \alpha'_j+1$, we have $\E[O_k] \leq 3$, so
\begin{align*}
	\E[O_{k+1}] &\leq \E[O_k]  - \frac{\E[O_k]}{n-k+1} + \frac{1}{n-k+1} \\
&\leq 3 \left(1 - \frac{1}{n-k+1} \right) + \frac{1}{n-k+1} < 3.
\end{align*}
 
If the partial map $m$ at the beginning of the $k^\upth$ step is bad, then the
partial map at the $(k+1)^\upst$ step will be bad if and only if the pairing dart is
bad.   If $m$ is not bad, then by Observation \ref{thm:obs2} the only mixed partial
face is the one containing the active dart $s_k$, and the only pairing dart that
results in a bad map is $t_{i_1}$ (in the notation of Observation \ref{thm:obs2}). 
Both cases of $m$ being bad and not bad are illustrated in Figure
\ref{fig:badmapmakingfromgood}.
\begin{figure}
	\begin{center}
    \includegraphics[scale=0.8]{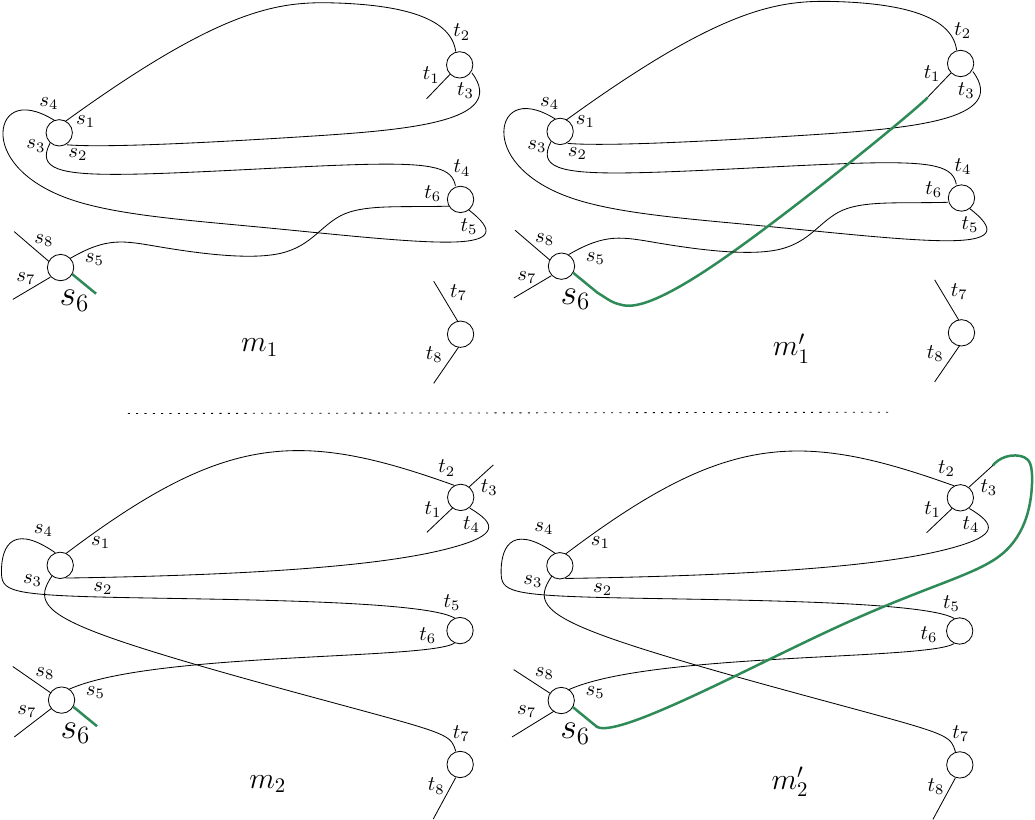}
	\end{center}
	\caption[hiya]{On the left are possible maps at the beginning of the $6^\upth$ iteration of the RPB, while the
		primed maps on the right are their output after the $6^\upth$ step is
		executed.  In both cases $s_6$ (green) is the active dart.  At top left,
		we have $u_{m_1} = (s_6\, s_7\, s_8)(t_1)(t_7\, t_8)$, so the map is bad.  The dart $s_6$ is paired with
		the bad dart $t_1$ to obtain $m_1'$, which is bad.  The reader can
		confirm that pairing $s_6$ with either dart $t_7$ or $t_8$ (neither dart
		is bad) instead of $t_1$ does not produce a bad map.\newline
		At bottom left, we have $u_{m_2} = (s_6\, s_7\, s_8\, t_3\, t_8)(t_1)$,
		so $m_2$ has a mixed partial face and is not bad.  Note the mixed
		partial face contains the active dart.  Furthermore, the dart $t_3$ plays the
		role of $t_{i_1}$ in Observation \ref{thm:obs2}, and if $s_6$ is paired
		with $t_3$, we get the bad map $m_2'$.  The reader can check that pairing
	$s_6$ with $t_1$ or $t_8$ instead of $t_3$ does not give a bad map.\label{fig:badmapmakingfromgood}
 }
\end{figure}
Since $k \neq \alpha'_j + 1$, we can use the inductive assumption and total probability to obtain
\begin{align*}
	Pr[b_{k+1}] &= Pr[b_{k+1} \mid b_k] Pr[b_k] + Pr[b_{k+1} \mid b_k^c] Pr[b_k^c] \\
&\leq Pr[b_k] \frac{\E[O_k \mid b_k]}{n-k+1} + Pr[b_k^c] \frac{1}{n-k+1} \\
&\leq Pr[b_k] \frac{\E[O_k \mid b_k]}{n-k+1} + Pr[b_k^c] \frac{1 + \E[O_k \mid b_k^c]}{n-k+1} \\
&\leq \frac{1 + \E[O_k]}{n-k+1} \leq \frac{4}{n-k+1}.
\end{align*}

Case 2:  $k = \alpha'_j$ for some $j$.
In this case the active dart $s_k$ is the last unpaired dart at its vertex.  If
the partial map at this step is not bad, then by Observation \ref{thm:obs} both $u^2(s_k)$ and $u^{-2}(s_k)$
are valid choices of pairing dart that add a bad dart to the resulting map.  Therefore we expect to
add $\frac{2}{n-k+1}$ bad darts in this case.  If the partial map at this step
is bad, then $s_k$ itself is a bad dart.  In this case we add one new bad dart
if and only if the pairing dart is in a partial face containing two unpaired
darts.  If we are at an early step and $\beta = 2^{n/2}$, then almost every dart
could satisfy this.  Therefore we use a rough estimate and assume any choice of
pairing dart will add a bad dart.  Regardless of if the partial map at this step
is bad or not, pairing $s_k$ with any bad dart will result in a map with one
fewer bad dart (the pairing dart will no longer be bad).  This gives 
\begin{align*}
	\E[O_{k+1}] &\leq \E[O_{k+1} \mid b_k] Pr[b_k] + \E[O_{k+1} \mid b_k^c]
	Pr[b_k^c]\\
		&\leq (\E[O_k \mid b_k] + 1) Pr[b_k] + \left(\E[O_k \mid b_k^c] + \frac{2}{n-k+1}\right)Pr[b_k^c] - \frac{\E[O_k]}{n-k+1} \\
		&= \E[O_k]\left( 1 - \frac{1}{n-k+1} \right) + Pr[b_k] \left( 1 -
			\frac{2}{n-k+1} \right) + \frac{2}{n-k+1}\\
		&\leq 3 \left( 1 - \frac{1}{n-k+1} \right) + \frac{4}{n-k+2} \left( 1 -
			\frac{2}{n-k+1} \right) + \frac{2}{n-k+1}\\
&= 3\left( 1 - \frac{1}{n-k+1} \right) + \frac{6}{n-k+1}- 
\frac{12}{(n-k+1)(n-k+2)}\\
&\leq 3 + \frac{3}{n-k+1}.
\end{align*}
At the next step, the active dart is $s_{k+1}$, and $k+1 = \alpha'_j+1$.  Then
by Observation \ref{thm:obs2} all of the darts at the vertex incident with $s_{k+1}$ are unpaired.  Therefore the partial map is necessarily bad.  This means $Pr[b_{k+1}] = 1$.

Case 3: We have $k = \alpha'_j + 1$ for some $j$.
In this case, again, the active dart $s_{k}$ is at a vertex
with all unpaired darts by Observation \ref{thm:obs2} and the map is bad.  Therefore there are no choices of pairing dart that
add a bad dart in $T^u$.  However, choosing a pairing dart that is bad will
remove this bad dart from the resulting map.  Therefore using induction we have
\begin{align*}
	\E[O_{k+1}] &\leq \left(1 - \frac{1}{n-k+1}\right) \E[O_k] \\
&\leq \left(1 - \frac{1}{n-k+1}\right) \left(3 + \frac{3}{n-k+2}\right)\\ 
&< 3 \left(\frac{n-k}{n-k+1}\right) \left(\frac{n-k+1}{n-k}\right) = 3.
\end{align*}

Since the partial map at the start of this step is bad, the partial map at the
next step will be bad if and only if the pairing dart is bad.  Therefore we
obtain 
\begin{align*}
	Pr[b_{k+1}] \leq \frac{\E[O_k]}{n-k+1} \leq \frac{3 + \frac{3}{n-k+2}}{n-k+1} \leq \frac{4}{n-k+1}.
\end{align*}
Note that when $n-3 \leq k \leq n$, the right hand side of the previous equation
is greater than or equal to 1, so it provides no interesting bound for $Pr[b_{k+1}]$.

This covers all the cases and hence completes the induction.
\end{proof}

We are now ready to give the lower bound.

\begin{proof} [Proof of the lower bound in Theorem \ref{thm:mainbound}]
	We estimate $\E[B_k]$ for each $k$ and then use \eqref{lem:bkdecomp}.  First
	observe that $\E[B_n] \geq 1$ and $\E[B_1] = 0$.  Second, for the special cases when $k = \alpha_j' + 1$ for some $j$, we find by Observations \ref{thm:obs} and \ref{thm:obs2} that the map at the
	beginning of the $k^\upth$ iteration is bad and that $\E[B_k] = 0$.

	Next, using total probability, we have 
	\begin{equation}\label{eq:totprob}
		\E[B_k] = \E[B_k \vert b_k] Pr[b_k] + \E[B_k \vert b_k^c] Pr[b_k^c] \geq
		\E[B_k \vert b_k^c] (1 - Pr[b_k]).
	\end{equation}
	To estimate the right hand side of \eqref{eq:totprob}, we use the estimate
on $Pr[b_k]$ from Lemma \ref{lem:baddartsestimate}, but note that estimate is only
useful when $k \leq n-3$ (otherwise the stated estimate on $Pr[b_k]$ is greater
than or equal to 1). Thus when $k = n-2$ or $n-1$, we use the coarse lower bound
of 0 for $\E[B_k]$.   Thus we have so far
\begin{align}
	&\E[B_1], \E[B_{n-2}] \textnormal{ and } \E[B_{n-1}] \geq 0,\notag\\
	&\E[B_n] \geq 1,\label{eq:cases}\\ 
	\textnormal{ and } &\E[B_k] = 0 \textnormal{ when } k = \alpha'_j + 1 \textnormal{ for some } j.\notag
\end{align}
For the remaining values of $k$, we estimate \eqref{eq:totprob} in two
cases:  when $k \neq \alpha_j'$ for all $j$ and when $k = \alpha'_j$ for some
$j$.  At the beginning of the $k^\upth$ step, we let $m$ be the partial map and
$v$ the vertex with the active dart $s_k$.  Since we are estimating $\E[B_k \vert b_k^c]$, we
assume $m$ is not bad.

    Case 1: $k \neq \alpha'_j$ for any $j$.  
	Then since $m$
	is not bad, by Observation \ref{thm:obs2} we have that the active dart $s_k$
	is in the lone mixed partial face, so $u^{-1}(s_k) \in T^u$, while $u(s_k)
	\in S^u$, from which it follows that $s_k$ is not bad.  But then from
	Observation \ref{thm:obs} it follows that the lone choice of pairing
	dart that adds a face is $u^{-1}(s_k)$.  So
	\begin{align*}
    \E[B_k | b_k^c] &\geq \frac{1}{n-k+1},
    \end{align*}
whence
\begin{align*}
	\E[B_k \vert b_k^c] (1 - Pr[b_k]) &\geq \frac{1}{n-k+1} \left(1 - \frac{4}{n-k+2} \right) \\
&= \frac{1}{n-k+1} - \frac{4}{(n-k+1)(n-k+2)}.
\end{align*}

    Case 2: $k = \alpha'_j$ for some $j$.  Since $m$ is not bad, by Observation
	\ref{thm:obs2} the active dart $s_k$ is in a mixed face, so $s_k$
    is not bad.  By
    Observation \ref{thm:obs} there are two choices of pairing dart that add a
    completed face: $u(s_k)$ and $u^{-1}(s_k)$.  Both of these may be possible as
    choices of pairing dart, since $s_k$ being the last unpaired dart at $v$
    implies $u(s_k), u^{-1}(s_k) \in T^u$.  As noted in Observation
    \ref{thm:obs}, if $u(s_k) = u^{-1}(s_k)$, then two faces are added.   Hence 
	\begin{align*}
    \E[B_k | b_k^c] &\geq \frac{2}{n-k+1} = \frac{1}{n-k} + \frac{1}{n-k+1} - \frac{1}{(n-k)(n-k+1)}.
    \end{align*}
This gives 
\begin{align*}
    \E[B_k \vert b_k^c] (1 - Pr[b_k]) &\geq \left(\frac{1}{n-k} + \frac{1}{n-k+1} - \frac{1}{(n-k)(n-k+1)}\right)\left(1- \frac{4}{n-k+2}\right) \\
&> \frac{1}{n-k+1}  + \frac{1}{n-k} - \frac{4}{(n-k+1)(n-k+2)} \\  
&\;\;\;\;\; - \frac{4}{(n-k)(n-k+2)} - \frac{1}{(n-k)(n-k+1)}\\
&> \frac{1}{n-k+1} - \frac{4}{(n-k+1)(n-k+2)} \\
&\;\;\;\;\; + \frac{1}{n-k} -
\frac{5}{(n-k)(n-k+1)}.
\end{align*}
Let 
\begin{align*}
	C&=\{\alpha'_j : j=1, \ldots, \ell(\alpha)-1 \textnormal{ and }
	\alpha'_j \leq n-3\}\\ 
	\textnormal{ and } D&=\{\alpha'_j+1 : j=1, \ldots, \ell(\alpha)-1 \textnormal{ and }
	\alpha'_j \leq n-4 \}, 
\end{align*}
and recall the identity
\begin{equation*}
	\sum_{k=i}^m \frac{1}{k(k+1)} = \frac{1}{i} - \frac{1}{m+1}.
\end{equation*}
Then, using \eqref{lem:bkdecomp}, \eqref{eq:totprob} and \eqref{eq:cases} and summing over $k$ gives:
    \begin{align*}
		\E[F_{\alpha,\beta}] & = \sum_{k=1}^n \E[B_k] = 1+\sum_{k=2}^{n-3} \E[B_k \vert b_k^c] (1 - Pr[b_k]) \\
&\geq 1 + \sum_{k=2 \atop k \notin C,D}^{n-3} \left(\frac{1}{n-k+1} -
\frac{4}{(n-k+1)(n-k+2)}\right)\\
& \;\;\;\;\;\; + \sum_{k \in C} \left(\frac{1}{n-k+1} -
\frac{4}{(n-k+1)(n-k+2)} + \frac{1}{n-k} - \frac{5}{(n-k)(n-k+1)} \right)\\
&= 1 + \sum_{k=2 \atop k \notin C,D}^{n-3} \left(\frac{1}{n-k+1} -
\frac{4}{(n-k+1)(n-k+2)}\right)\\
& \;\;\;\;\;\; + \sum_{k \in C,D} \left(\frac{1}{n-k+1} -
\frac{5}{(n-k+1)(n-k+2)}  \right)  \\
&> 1 + \sum_{k=2}^{n-3} \frac{1}{n-k+1}  - 5\sum_{k =2}^{n-3} \frac{1}{(n-k+1)(n-k+2)}\\ 
&= H_n - \frac{1}{2} - \frac{1}{3} - \frac{1}{n}  - 5\left(\frac{1}{4} -
\frac{1}{n}\right) = H_n - \frac{25}{12} + \frac{4}{n} > H_n -3.\qedhere
	\end{align*}
\end{proof}

\section*{Acknowledgements}

Both authors would like to thank \'E. Fusy for helpful comments. In particular his
comments on Section \ref{sec:rpa} helped us improve proofs for the theorems
of that section.

\bibliographystyle{alpha}
\bibliography{cites}

\end{document}